\documentclass[a4paper,10pt]{article}

\usepackage[utf8]{inputenc}
\usepackage[english]{babel}
\usepackage[babel]{csquotes}
\usepackage{fullpage}
\usepackage{amsfonts, amsthm, amsmath}
\usepackage{enumerate}
\usepackage{bbm}       %for indicator function \mathbbm{1}
\usepackage[version=3]{mhchem}    %for writing reaction networks with reversible reactions
\usepackage{tikz}      %for more complicated reaction networks
\usetikzlibrary{automata,positioning}

\numberwithin{equation}{section}

\newtheorem{theorem}{Theorem}[section]
\newtheorem{corollary}[theorem]{Corollary}
\newtheorem{lemma}[theorem]{Lemma}
\newtheorem{proposition}[theorem]{Proposition}
\newtheorem{definition}{Definition}
\newtheorem{appxlemma}{Lemma}[section]

\theoremstyle{definition}
\newtheorem{ex}{Example}

\newcommand{\Sp}{\mathcal{X}}
\newcommand{\C}{\mathcal{C}}
\newcommand{\R}{\mathcal{R}}
\newcommand{\G}{\mathcal{G}}
\newcommand{\NN}{\mathbb{N}}
\newcommand{\RR}{\mathbb{R}}
\newcommand{\pr}[1]{\left(#1\right)}

\DeclareMathOperator{\spann}{span}
\DeclareMathOperator{\kernel}{Ker}
\DeclareMathOperator{\supp}{supp}

\title{Product-form Poisson-like distributions and complex balanced  reaction  systems}
\author{Daniele Cappelletti\footnotemark[1] \and Carsten Wiuf\footnotemark[1]}

\begin{document}

 \footnotetext[1]{Department of Mathematical Sciences, University of Copenhagen, Copenhagen, Denmark.  The authors are supported by  the Danish Research Councils.}

 \tikzset{every node/.style={auto}}
 \tikzset{every state/.style={rectangle, minimum size=0pt, draw=none, font=\normalsize}}

 \maketitle

\begin{abstract}
Stochastic reaction networks are dynamical models of biochemical reaction systems and form a particular class of continuous-time Markov chains on $\NN^n$. Here we provide a fundamental characterisation   that connects structural properties of a network to its dynamical features. Specifically, we define the notion of `stochastically complex balanced systems' in terms of the network's stationary distribution and provide a characterisation of stochastically complex balanced systems, parallel to that established in the 70-80ies for deterministic  reaction networks.  Additionally, we establish that a network is stochastically complex balanced  if and only if an associated deterministic network is  complex balanced (in the deterministic sense), thereby proving a strong link between the theory of stochastic and   deterministic networks.
Further, we prove a stochastic version of the `deficiency zero theorem' and show that  any (not only complex balanced)  deficiency zero reaction network has a product-form Poisson-like stationary distribution on all irreducible components. 
 Finally, we provide  sufficient conditions for when a product-form Poisson-like distribution on a single (or all) component(s) implies the network is complex balanced, and explore the possibility to characterise complex balanced systems in terms of  product-form Poisson-like stationary distributions.
\end{abstract}

\section{Introduction}\label{sec:introduction}

Improved experimental techniques have made it possible to measure molecular fluctuations at a small scale, creating a need for a stochastic description of molecular data \cite{zechner:scalable,hey:switch}. Typically, biochemical reaction networks are  modelled  as  deterministic systems of ordinary differential equations (ODEs), but these models assume the  individual species are in high concentrations and do not allow for stochastic fluctuation. An alternative is stochastic models based on continuous-time Markov chains \cite{kurtz:classical_scaling, kurtz:strong,ingram:nonidentifiability,anderson:design, anderson:book, hey:switch}.  As an example of a stochastic reaction system, consider
 \begin{equation}\label{eq:reactions}
 A+B \cee{<=>[\kappa_1][\kappa_2]} 2C,
  \end{equation}
  where $\kappa_1,\kappa_2$ are positive reaction constants.
 The network consists of three chemical species $A$, $B$ and $C$ and two reactions. Each occurrence of a reaction modifies the species counts, for example, when the reaction $A+B\to2C$ takes places,  the amount of $A$ and $B$ molecules are each decreased by one, while two molecules of $C$ are created.   The species counts are modelled as a continuous-time Markov chain, where the transitions are  single occurrences of  reactions with  transition rates 
   \begin{align*}
   \lambda_1(x) &=\kappa_1x_Ax_B, \quad
   \lambda_2(x) =\kappa_2 x_C(x_C-1),
   \end{align*}
 and $x=(x_A,x_B,x_C)$ are the species counts \cite{anderson:design}. When modelled deterministically, the concentrations (rather than the counts) of the species  change according to an ODE system.

In  a classical paper \cite{kurtz:classical_scaling},  Kurtz explored the relationship between deterministic and stochastic reaction systems, using a scaling argument -- large volume limit -- to link the dynamical behaviour of the two types of systems to each other. 
Other,  mainly recent work, also points to close connections between the two types of systems \cite{whittle:systems,anderson:product-form, anderson:ACR, anderson:lyapunov, ball:asymptotic, kurtz:rescale}. In this paper we explore this relationship further. 

A fundamental link between structural network properties and dynamical features of deterministic reaction networks has been known since the 1970s and 1980s  with the work of Horn, Jackson and Feinberg \cite{horn:general_mass_action, feinberg:review}. Specifically, their theory concerns the existence and uniqueness of equilibria in \emph{complex balanced} systems, with  the `deficiency zero theorem' playing a central role in this context. Complex balanced systems were called cyclic balanced systems by Boltzmann. They have attractive analytical and physical properties; for example a (pseudo-)entropy might be defined which increases along all trajectories (Boltzmann's H-theorem) \cite{boltzmann,horn:general_mass_action}.

A parallel theory for the stochastic regime is not available, and the very concept of ``complex balanced'' does not currently have a stochastic counterpart.  In this paper we develop a theory to fill this gap. We  define \emph{stochastically complex balanced} systems  through properties of the stationary distribution, and we prove results for stochastic reaction networks that are in direct correspondence with the results for deterministic models. In particular, we prove a parallel statement of the deficiency zero theorem and show that all deficiency zero reaction networks have product-form Poisson-like stationary distributions, irrespectively whether they are complexed balanced or not. In fact, in the non-complexed balanced case, the network is complex balanced on the boundary of the state space. 
  
 A second target of our study concerns product-form stationary distributions. Such distributions are computationally and analytically tractable and  appear in  many areas of applied probability, such as,  queueing theory \cite{jackson:productform,kelly:reversibility},  Petri Net theory \cite{marin:analysis}, and stochastic reaction network theory \cite{whittle:systems,mairesse:definciencyzero,anderson:product-form}.
 Specifically, a complex balanced mass-action network has a product-form Poisson-like stationary distribution on every irreducible component \cite{mairesse:definciencyzero,anderson:product-form}.
 As an example, the stationary distribution of the complex balanced reaction system \eqref{eq:reactions} is
 $$\pi_\Gamma(x)=M_\Gamma\frac{\kappa_1^{x_A}\kappa_2^{x_B}\kappa_1^{x_C}}{x_A!x_B!x_C!}\quad\text{for }x\in\Gamma,$$
 where $\Gamma=\{x\in\NN^3\colon x_A+x_B+2x_C=\theta\}$ is an irreducible component of the state space $\NN^3$ and $M_\Gamma$ is a normalising constant. 
 
 We expand the above result on mass-action systems and give general conditions under which the converse statement is true. In particular, we are interested in providing a structural characterisation of the networks with product-form Poisson-like stationary distributions. However, this class of networks is strictly larger than that of complex balanced networks, and a full characterisation seems hard to achieve. We illustrate this with examples.

\section{Background}
 
 We first introduce the necessary notation and background material; see \cite{anderson:design,feinberg:review,erdi:mathematical_models} for general references. We assume standard knowledge about continuous-time Markov chains.
 
 \subsection{Notation} 
 
 We let $\RR$, $\RR_0$ and $\RR_+$ be the real, the non-negative real and the positive real numbers, respectively. Also let $\NN$ be the natural numbers including 0.
 
 For any real number $a\in\RR$, $|a|$  denotes the absolute value of $a$. Moreover, for any vector $v\in\RR^p$, we let $v_i$ be the $i$th component of $v$, $\|v\|$ the Euclidean norm, and $\|v\|_\infty$ the infinity norm, that is, $\|v\|_\infty=\max_i |v_i|$. 
 For two vectors $v,w\in\RR^p$,  we write $v<w$ (resp.\,$v>w$) and $v\leq w$ (resp.\,$v\ge w$),   if the inequality holds component-wise. Further, we define $\mathbbm{1}_{\{v\leq w\}}$ to be one if $v\leq w$, and zero otherwise, and similarly for the other inequalities. If $v>0$ then $v$ is said to be positive. Finally,  $\supp v$ denotes the index set of the non-zero components. For example, if $v=(0,1,1)$ then $\supp v=\{2,3\}$.
 
 If $x\in\RR^q_0$ and $v\in \NN^q$, we define
 $$x^v=\prod_{i=1}^q x_i^{v_i},\quad \text{and}\quad v!=\prod_{i=1}^q v_i!,$$
 with the conventions that $0!=1$ and $0^0=1$. 

 \subsection{Reaction networks}  
 
 A reaction network is a triple $\G=(\Sp,\C,\R)$, where
$\Sp=\{S_1,S_2,\dots,S_n\}$ is a set of $n$ species, $\C$ is a set of $m$ complexes, and $\R\subseteq \C\times\C$ is a set of $k$ reactions, such that  $(y,y)\notin\R$ for all  $y\in\C$. The complexes are  linear combinations of species on $\NN$, identified as vectors in $\RR^n$.  A reaction $(y,y')\in\R$ is denoted by $y\to y'$.  We require that every  species is part of at least one complex, and that every complex is part of at least one reaction. In this way, there are no ``superfluous'' species or complexes and $\G$ is completely determined by the set of reactions $\R$, which we allow to be empty. In  \eqref{eq:reactions}, there are $n=3$ species ($A,B,C$), $m=2$ complexes ($A+B, 2C$), and $k=2$ reactions. 

Given a reaction network $\G$,  the \emph{reaction graph} of $\G$ is the directed graph with node set $\C$ and edge set $\R$. We let $\ell$ be the number of linkage classes (connected components) of the reaction graph. A reaction $y\to y'\in\R$ is \emph{terminal} if any directed path that starts with $y\to y'$ is contained in a closed directed path. We let $\R^*$ be the set of terminal reactions.
 
A reaction network $\G$ is \emph{weakly reversible}, if every reaction is terminal. The network in \eqref{eq:reactions} is weakly reversible, since both  reactions are terminal.
 
  The \emph{stoichiometric subspace} of $\G$ is the linear subspace of   $\RR^n$ given by
 $$S=\spann(y'-y\vert y\to y'\in\R).$$
 For $v\in\RR^n$, the sets $(v+S)\cap\RR^n_0$ are called the \emph{stoichiometric compatibility classes} of $\G$ (Fig.\ $1A$). For the network in \eqref{eq:reactions}, $S=\spann( (-1,1,0),(0,1,-1))\subset\RR^3$, which is 2-dimensional.
 
 \subsection{Dynamical systems}
 
 We will consider a reaction network $\G$ either as a deterministic dynamical system on the continuous space $\RR^n_0$, or as a stochastic dynamical system on the discrete space $\NN^n$. 

 In the deterministic case, the evolution of the species concentrations $z=z(t)\in\RR^n_0$  at time $t$ is modelled as the solution to the ODE
 \begin{equation}\label{eq:ODE_mass_action}
  \frac{d z}{dt}=\sum_{y\to y'\in\R}(y'-y)\lambda_{y\to y'}(z),
 \end{equation}
 for some functions $\lambda_{y\to y'}\colon\RR^n_0\to\RR_0$ and an initial condition $z(0)\in\RR^n_0$. We require that the functions $\lambda_{y\to y'}$ are continuously differentiable, and that $\lambda_{y\to y'}(z)>0$ if and only if $\supp y\subseteq\supp z$. Such functions are called \emph{rate functions}, they constitute a \emph{deterministic kinetics} $K$ for $\G$, and the pair $(\G,K)$ is called a \emph{deterministic reaction system}. If  $\lambda_{y\to y'}(z)=\kappa_{y\to y'}z^{y}$ for all reactions, then the constants $\kappa_{y\to y'}$ are referred to as \emph{rate constants} and the modelling regime is referred to as \emph{deterministic mass-action kinetics}. In this case, the pair $(\G,\kappa)$ is called a \emph{deterministic mass-action system}, where $\kappa\in\RR_+^k$ is the vector of rate constants.
 
 In the stochastic setting, the evolution of the species counts  $X(t)\in\NN^n$  at time $t$ is modelled as a continuous-time Markov chain with state space $\NN^n$. At any state $x\in\NN^n$, the states that can be reached in one step are  $x+y'-y$ for $y\to y'\in\R$, with transition rates $\lambda_{y\to y'}(x)$. The functions $\lambda_{y\to y'}\colon\NN^n\to\RR_0$ are called \emph{rate functions}, and we require that $\lambda_{y\to y'}(x)>0$ if and only if $x\geq y$. A choice of these functions constitute a \emph{stochastic kinetics} $K$ for $\G$ and the pair $(\G,K)$ is called a \emph{stochastic reaction system}.
 If  the reaction $y\to y'$ occurs at time $t$, then the new state  is
 $$X(t)=X(t-)+y'-y,$$
 where $X(t-)$ denotes the previous state.  If for any reaction $y\to y'\in\R$
 $$\lambda_{y\to y'}(x)=\kappa_{y\to y'}\frac{x!}{(x-y)!}\mathbbm{1}_{\{x\geq y\}}.$$
 then the constants $\kappa_{y\to y'}$ are known as \emph{rate constants}, as in the deterministic case, and the modelling regime is referred to as \emph{stochastic mass-action kinetics}. The pair $(\G,\kappa)$ is, in this case, called a \emph{stochastic mass-action system}.
  
 The evolution of the stochastic as well as the deterministic reaction system is confined to the stoichiometric compatibility classes, 
 $$z(t)\in (z(0)+S)\cap\RR^n_0  \quad\text{and}\quad X(t)\in (X(0)+S)\cap\RR^n_0.$$
 In fact, $X(t)\in (X(0)+S)\cap\NN^n$, as $X(t)$ takes values in $\NN^n$. 
  \begin{definition}
Let $\G=(\Sp,\C,\R)$ be a reaction network.
 \begin{enumerate}[a)]
 \item  A reaction network $\G'=(\Sp',\C',\R')$ is a \emph{subnetwork} of $\G$ if  $\R'\subseteq\R$. In this case, it follows that $\Sp'\subseteq\Sp$ and $\C'\subseteq \C$.
 \item A system $(\G',K')$, deterministic or stochastic, is a  \emph{subsystem} of a system $(\G,K)$  if  $\G'$ is a subnetwork of $\G$ and the rate functions agree on the reactions in $\R'$.
 \item The subnetwork $\G^*$ given by the set of terminal reactions $\R^*$ is the \emph{terminal network} of $\G$. We denote $\G^*=(\Sp^*, \C^*, \R^*)$. Furthermore, the subsystem $(\G^*,K^*)$ of $(\G,K)$ is called the \emph{terminal system} of $(\G,K)$.
 \end{enumerate}
 \end{definition}
 \begin{definition}
  The connected components of the reaction graph of the terminal network of $\G$ are called \emph{terminal strongly connected component} of $\G$. For any complex $y$ in $\C^*$, we denote by $(\G_y,K_y)$ the subsystem of $\G$ whose reaction graph is the terminal strongly connected component containing $y$ as node.
 \end{definition}
 
  As an example, consider the mass-action system
 $$2A\cee{<=>[\kappa_1][\kappa_2]}2B\cee{<-[\kappa_3]}A\cee{->[\kappa_4]}0\cee{<=>[\kappa_5][\kappa_6]}C.$$
 Here, there are two terminal strongly connected components, which are $2A\rightleftharpoons 2B$ and $0\rightleftharpoons C$.
 In particular, $(\G_{2A},K_{2A})$ is equal to $(\G_{2B},K_{2B})$ and is given by
 $$2A\cee{<=>[\kappa_1][\kappa_2]}2B.$$
 Finally, if $(\G,\kappa)$ is a mass-action system, any subsystems $(\G',K')$ is a mass-action systems as well and can be denoted by $(\G',\kappa')$.

\section{Deterministic reaction systems}

In this section we will recapitulate the known characterisation of existence and uniqueness of  positive equilibria in complex balanced systems and the connection between complex balanced systems and  deficiency zero reaction networks. As we will show in the subsequent section, this characterisation can be fully translated into a similar characterisation for stochastic reaction networks.

 \subsection{Complex balanced systems}
 
 We start with a definition.
 
 \begin{definition}\label{def:cb}
 A deterministic reaction system $(\G,K)$ is said to be \emph{complex balanced} if there exists a positive \emph{complex balanced equilibrium}, that is, a positive equilibrium point $c\in\RR^n_+$ for the system \eqref{eq:ODE_mass_action}, such that
 \begin{equation}\label{eq:cb}
 \sum_{y'\in\C}\lambda_{y\to y'}(c)=\sum_{y'\in\C}\lambda_{y'\to y}(c)\quad \text{for all }y\in\C.
 \end{equation}
 \end{definition}
The name `complex balanced' refers to the fact that the flow, at equilibrium, entering into the complex $y$  equals the flow exiting from the complex. As an example, the mass-action system in \eqref{eq:reactions} is complex balanced for any choice of $(\kappa_1, \kappa_2)$ and $c=(\kappa_2,\kappa_1,\kappa_1)$ is a complex balanced equilibrium. The class of complex balanced systems is an extension of the class of detailed balanced mass-action systems  \cite{horn:general_mass_action,feinberg:review}.
 
 For mass-action systems, \eqref{eq:cb} becomes
 \begin{equation}\label{eq:cb_mass-action}
 \sum_{y'\in\C}\kappa_{y\to y'}c^y=\sum_{y'\in\C}\kappa_{y'\to y}c^{y'}\quad \text{for all }y\in\C,
 \end{equation}
 with the convention that $k_{y\to y'}=0$ if $y\to y'\not\in\R$.
   
 In the case of mass-action kinetics, we extend Definition \ref{def:cb} to the stochastic case, by saying that a stochastic mass-action system $(\G,\kappa)$ is complex balanced if the deterministic mass-action system $(\G,\kappa)$ is complex balanced. We might therefore refer to complex balanced mass-action systems without specifying whether they are stochastically or deterministically modelled. 
 
The next theorem is a slight generalization of a classical result \cite{horn:general_mass_action}, which provides the backbone for the further characterisation. The generalization includes a property of non-negative equilibria
 \begin{theorem}\label{thm:complex_balanced}
  If a deterministic reaction system $(\G,K)$ is complex balanced, then $\G$ is weakly reversible. Moreover, if $K$ is mass-action kinetics, then all equilibria are complex balanced, that is, they fulfil \eqref{eq:cb_mass-action}. Moreover, there exists exactly one positive equilibrium in each stoichiometric compatibility class, which is locally asymptotically stable.
 \end{theorem}
 
 As we are not aware of a proof of this more general formulation, we provide one in Appendix \ref{sec:proofs}.
 
 \subsection{Deficiency zero statements}
 
 The \emph{deficiency} plays an important role in the study of complex balanced systems.  The deficiency of $\G$ is defined as
 $$\delta=m-\ell-s,$$
 where $m$ is the cardinality of $\C$, $\ell$ is the number of linkage classes of the reaction graph of $\G$ and $s$ is the dimension of the stoichiometric subspace $S$ \cite{horn:general_mass_action}. The definition hides the geometrical interpretation of the deficiency, which we now will explore. 
 
 Let $\{e_y\}_{y\in\C}$ be a basis of $\RR^m$. Further, define
 $$d_{y\to y'}=e_{y'}-e_{y}\quad\text{and}\quad\xi_{y\to y'}=y'-y$$
 for $y\to y'\in\R$. Let $D=\spann(d_{y\to y'}\vert y\to y'\in\R)$. Then $\dim D=m-\ell$ \cite{horn:general_mass_action}.
 
 The space $D$ is linearly isomorphic to the stoichiometric subspace $S$ if and only if $\delta=0$. Specifically, consider the homomorphism
 \begin{equation}\label{eq:def_phi}
  \begin{array}{rrcl}
    \varphi\colon & \RR^{m} & \to     & \RR^n \\
               & e_y   & \mapsto & y.
  \end{array}
 \end{equation}
 For  $y\to y'\in\R$, we have $\varphi(d_{y\to y'})=\xi_{y\to y'}$ and $\varphi_{|D}\colon D\to S$  is thus a surjective homomorphism. Therefore,
 \begin{equation}\label{eq:deficiency_kernel}
  \dim\kernel \varphi_{|D}=\dim D-s=m-\ell-s=\delta,
 \end{equation}
 which implies that $\varphi_{|D}$ is an isomorphism if and only if $\delta=0$. It further follows that the deficiency is a non-negative number.
 
 We state here a useful Lemma on the deficiency of subnetworks.
 
 \begin{lemma}\label{lem:subnet}
 Let $\G$ be a reaction network with deficiency $\delta$. Then, the deficiency of any subnetwork of $\G$ is smaller than or equal to $\delta$.
 \end{lemma}
 \begin{proof}
  Let $\R'\subseteq\R$ and let $\G'$ be the corresponding subnetwork with deficiency $\delta'$. 
  Further, let $D'$ and $S'$ be the equivalent of $D$ and $S$ for $\G'$, respectively. By \eqref{eq:deficiency_kernel} and since $D'$ is a subspace of $D$, we have  $\delta'=\dim\kernel \varphi_{|D'}\leq \dim\kernel \varphi_{|D}=\delta,$
  which concludes the proof.
 \end{proof}

 We next state two classical results which elucidate the connection between complex balanced systems and deficiency zero systems. A proof of the first and of the second result can be found in \cite{horn:general_mass_action} and in \cite{feinberg:review}, respectively. The results draw a connection between graphical and dynamical properties of a network. Theorem \ref{thm:equilibrium_boundary} is given here in a  wider formulation than in \cite{feinberg:review} (see Appendix \ref{sec:proofs} for a proof). 

 \begin{theorem}\label{thm:deficiency_zero_iff}
  The mass-action system $(\G,\kappa)$ is complex balanced for any choice of $\kappa$ if and only if $\G$ is weakly reversible and its deficiency is zero.
 \end{theorem}

 \begin{theorem}\label{thm:equilibrium_boundary}
  Consider a deterministic reaction system $(\G,K)$, and assume that the deficiency of $\G$ is zero. If $x\in\RR_0^n$ is an equilibrium point and $y\to y'\in\R$, then $\supp y\subseteq \supp x$ only if $y\to y'$ is terminal.
  Moreover, if $K$ is mass-action kinetics with rate constants $\kappa$ and $\supp y\subseteq \supp x$, then the projection of $x$ onto the species space of $\G_y$ is a complex balanced equilibrium of $(\G_y,\kappa_y)$.
 \end{theorem}

It follows from Theorem \ref{thm:equilibrium_boundary} that an equilibrium point satisfies \eqref{eq:cb_mass-action} for the terminal system, though it is not necessarily a positive equilibrium of  $(\G^*,\kappa^*)$.% and cannot therefore be considered complex balanced.
The deficiency zero theorem, in the following formulation, is a consequence of the three previous theorems:

 \begin{theorem}[Deficiency zero theorem]\label{thm:deficiency_zero}
 Consider a deterministic reaction system $(\G,K)$  for which the deficiency is zero. Then the following statements hold:
  \begin{enumerate}[i)]
   \item if $\G$ is not weakly reversible, then there exists no positive equilibria;
   \item if $\G$ is weakly reversible and $K$ is mass-action kinetics, then there exists within each stoichiometric compatibility class a unique positive equilibrium, which is asymptotically stable.
  \end{enumerate}
 \end{theorem}
 
 The original formulation is richer than the one presented here \cite{feinberg:review}.
 
\section{Stochastic reaction systems}
 
\subsection{Classification of states and sets}

\begin{figure}
 \begin{center}
   \includegraphics[width=0.8\textwidth]{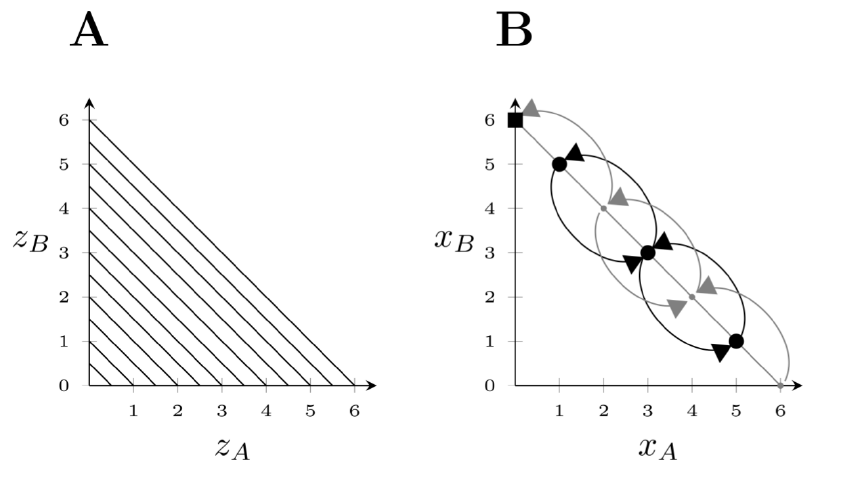}
  \caption{The figure shows some features of the reaction network $2A\to2B$, and $A+3B\to3A+B$. (A) The stoichiometric compatibility classes are  of the form $\{(z_A,z_B)\colon z_A+z_B=\mathrm{const.}\}$. (B) The two irreducible components on $\{(x_A,x_B)\colon x_A+x_B=6\}$ are shown (black circles and square), together with the possible transitions between the  states. All states within a component are accessible from each other. The ``square'' component has no active reactions, both reactions are active on the ``black circles'' component. The grey states are transient states which are not in any irreducible component.
  }
   \label{figure}
 \end{center}
\end{figure} 

 To characterise the stochastic dynamics we introduce the following terminology.
 \begin{definition}
  Let $\G=(\Sp,\C,\R)$ be a reaction network. 
  \begin{enumerate}[a)]
   \item A reaction $y\to y'\in\R$ is \emph{active} on $x\in\NN^n$ if $x\geq y$.
   \item A state $u\in\NN^n$ is \emph{accessible} from a state $x\in\NN^n$ if there is a sequence of $q\ge 0$ reactions $(y_j\to y_j')_{j=1,\ldots,q}$ such that
    \begin{enumerate}
     \item[(i)] $u=x+\sum_{j=1}^q (y'_j-y_j)$,
     \item[(ii)] $y_h\to y_h'$ is active on $x+\sum_{j=1}^{h-1}(y'_j-y_j)$ for all $1< h\leq q$.
    \end{enumerate}
  \end{enumerate}
 \end{definition} 
 \begin{definition}
  Let $\G$ be a reaction network. A non-empty set $\Gamma\subseteq\NN^n$ is an \emph{irreducible component} of $\G$ if for all $x\in\Gamma$ and all $u\in\NN^n$, $u$ is accessible from $x$ if and only if $u\in\Gamma$.
 \end{definition}

\begin{definition} 
 A reaction network $\G$ is \emph{essential} if the state space is a union of irreducible components. A reaction network $\G$ is \emph{almost essential} if the state space is a union of irreducible components except for a finite  number of states.
\end{definition}
% `Irreducible' and `essential' are standard terms in Markov chain theory. 
% A reaction network is essential if and only if every state of the associated Markov chain is `essential', which means that there exists always a positive probability to visit the state in the future, given that it was visited in the past \cite{gusak:theory}.

An essential network is also almost essential. A weakly reversible reaction network is essential \cite{craciun:dynamical}. Conditions for being essential can be found in \cite{craciun:dynamical, gupta:determining}.  Any irreducible component is  contained in some stoichiometric compatibility class, and a stoichiometric compatibility class may contain several irreducible components (Fig.\ $1B$).
 
\subsection{Stationary distribution}

 The stationary distribution $\pi_\Gamma$ on an irreducible component $\Gamma$ is unique, if it exists. It is  characterised by the \emph{master equation} \cite{anderson:design}:
 \begin{equation}\label{eq:master}
  \sum_{y\to y'\in\R}\pi_\Gamma(x+y-y')\lambda_{y\to y'}(x+y-y') = \pi_\Gamma(x)\sum_{y\to y'\in\R}\lambda_{y\to y'}(x),  
 \end{equation}
 for all $x\in\Gamma$. Let $X(t)$ denote the stochastic process associated with the system.  If  $X(t_0)$ follows the law of $\pi_\Gamma$ at time $t_0$, then the distribution of $X(t)$ is $\pi_\Gamma$ for all future times $t\geq t_0$. In this sense, the stationary distribution describes a state of equilibrium of the system. Moreover, if $\pi_\Gamma$ exists, then
 \begin{equation}\label{eq:limit_stationary}
  \lim_{t\to\infty} P(X(t)\in A)=\pi_\Gamma(A)\quad\text{for any }A\subseteq\Gamma,
 \end{equation}
 provided that $X(0)\in\Gamma$ with probability one. As discussed in Section \ref{sec:introduction}, a connection between mass-action complex balanced systems and their stationary distribution has been made in \cite{anderson:product-form}:
 \begin{theorem}\label{thm:anderson}
  Let $(\G,\kappa)$ be a complex balanced mass-action system. Then, there exists a unique stationary distribution on every irreducible component $\Gamma$, and it is of the form
  \begin{equation}\label{eq:definition_stat_prob}
   \pi_\Gamma(x)=M^{c}_\Gamma\prod_{i=1}^n\frac{c_i^{x_i}}{x_i!}\quad\text{for }x\in\Gamma,
  \end{equation}
  where $c$ is a positive complex balanced equilibrium of $(\G,\kappa)$ and $M^{c}_\Gamma$ is a normalising constant.
 \end{theorem}
%  The result in \cite{anderson:product-form} assumes deficiency zero, but the proof only uses the fact that $c$ is a positive complex balanced equilibrium, and therefore it holds in the broader setting of complex balanced systems. In fact, it is stated in this more general form in \cite{anderson:book}.
%  Our results will expand Theorem \ref{thm:anderson}.

\subsection{Parallel theorems for stochastic mass-action systems}

In this section we  derive  stochastic statements corresponding to  Theorem \ref{thm:complex_balanced}-\ref{thm:deficiency_zero}. Some of the proofs are deferred to Appendix \ref{sec:proofs}. We begin with a definition.

 \begin{definition}
   For an irreducible component $\Gamma$, the set $\R_\Gamma$ of \emph{active reactions} on $\Gamma$ consists of the reactions $y\to y'\in\R$ that are active on some $x\in\Gamma$. The subnetwork $\G_\Gamma=(\Sp_\Gamma,\C_\Gamma,\R_\Gamma)$ is called the \emph{$\Gamma$-network} of $\G$ and the subsystem $(\G_\Gamma,K_{\Gamma})$ of $(\G,K)$ is called the \emph{$\Gamma$-system} of $(\G,K)$.
 \end{definition}
 
The reactions that are active on $\Gamma$ determine the dynamics of the stochastic system on $\Gamma$. To study the stationary distributions, it is therefore convenient to analyse the $\Gamma$-systems. Note that $\R_\Gamma$ is empty if and only if $\Gamma$ consists of a single state.
 
 As an example, consider the  deficiency zero network, 
 $$C\rightleftharpoons D,\quad 2A\rightleftharpoons 2B,\quad A\to 0.$$
 All  molecules of $A$ and $B$ are irreversibly consumed through $A\to0$ and $2B\to 2A$, thus the only active reactions on an irreducible component $\Gamma\not=\{0\}$ are $C\rightleftharpoons D$. 
 The $\Gamma$-network is therefore $C\rightleftharpoons D$, which differs from the terminal system $C\rightleftharpoons D$, $2A\rightleftharpoons 2B$.
 The next proposition states that for a deficiency zero reaction network $\R_\Gamma\subseteq\R^*$ for any irreducible component $\Gamma$. Note that Proposition \ref{prop:terminal_strong_linkage_classes} does not hold in general, for example,
 $$A\to B,\quad 2B\to 2A$$
 has $\R_{\Gamma}=\R$ for any $\Gamma\neq\{0\}, \{(0,1)\}$, while $\R^*=\emptyset$.
 \begin{proposition}\label{prop:terminal_strong_linkage_classes}
  Let $\G$ be a reaction network and $\Gamma$ an irreducible component such that $\G_\Gamma$ has deficiency zero. Then, $\G_\Gamma$ is a subnetwork of $\G^*$. In particular, this is true if the deficiency of $\G$ is zero.
 \end{proposition}

 See Appendix \ref{sec:proofs} for a proof. Proposition \ref{prop:terminal_strong_linkage_classes} can be useful because $\R_\Gamma$ might be difficult to find, especially if there are  many complexes.  On the other hand, terminal reactions are easily identified by means of the reaction graph. The next definitions are inspired by Definition \ref{def:cb}.
 
  \begin{definition}\label{defi:complex_balanced_stationary_distribution}
  Let $(\G, K)$ be a stochastic reaction system. A stationary distribution $\pi_\Gamma$ on an irreducible component $\Gamma$ is said to be \emph{complex balanced} if
  \begin{equation}\label{eq:stochastic_complex_balanced}
   \sum_{y\in\C_\Gamma}\pi_\Gamma(x-y'+y)\lambda_{y\to y'}(x-y'+y)=\sum_{y\in\C_\Gamma}\pi_\Gamma(x)\lambda_{y'\to y}(x)\quad\forall y'\in\C_\Gamma, x\in\Gamma. 
  \end{equation}
 \end{definition}
  For a mass-action system, \eqref{eq:stochastic_complex_balanced} becomes
  $$\sum_{y\in\C_\Gamma}\pi_\Gamma(x-y'+y)\kappa_{y\to y'}\frac{(x-y'+y)!}{(x-y')!}\mathbbm{1}_{\{x\geq y'\}}=\sum_{y\in\C_\Gamma}\pi_\Gamma(x)\kappa_{y'\to y}\frac{x!}{(x-y')!}\mathbbm{1}_{\{x\geq y'\}}$$
  for any $y'\in\C_\Gamma$ and $x\in\Gamma$, with the convention that $k_{y\to y'}=0$ if $y\to y'\not\in\R_\Gamma$. In developing the theory for complex balanced equilibria in the deterministic setting, an important role is played by requiring positivity of the complex balanced equilibrium. Our aim is to introduce a similar concept for the stochastic systems. In the deterministic setting, if a state $z\in\RR^n$ is positive then every rate function calculated on $z$ is positive. We find inspiration from this to give the next definition:
  \begin{definition}
   An irreducible component $\Gamma$ is \emph{positive} if $\G_\Gamma=\G$.
  \end{definition}
  
  Equivalently, an irreducible component $\Gamma$ is \emph{positive} if all reactions are active on $\Gamma$. The next definition follows naturally by analogy with the deterministic setting.

 \begin{definition}\label{def:stoch_cb}
  A stochastic reaction system $(\G,K)$ is said to be \emph{stochastically complex balanced} if there exists a complex balanced stationary distribution on a positive irreducible component.
 \end{definition}
 
 If $\Gamma$ is positive, then $\C_\Gamma=\C$ and a complex balanced stationary distribution on $\Gamma$ satisfies \eqref{eq:stochastic_complex_balanced} with $\C_\Gamma$ replaced by $\C$.  Note the similarity between Definition \ref{def:stoch_cb} and the definition of a complex balance equilibrium (Definition \ref{def:cb}): the positivity of $\Gamma$ plays the role of the positivity in Definition \ref{def:cb}. Also note the close similarity between \eqref{eq:cb} and \eqref{eq:stochastic_complex_balanced}.
 
 \begin{theorem}\label{thm:stochastic_complex_balanced}
  Let $(\G,K)$ be a stochastic reaction system, and let $\Gamma$ be an irreducible component. If there exists a complex balanced stationary distribution $\pi_\Gamma$ on $\Gamma$ then $\G_\Gamma$ is weakly reversible. Moreover, if $K$ is mass-action kinetics with rate constants $\kappa$, there exists a complex balanced stationary distribution $\pi_\Gamma$ on $\Gamma$ if and only if the $\Gamma$-system of $(\G,\kappa)$ is complex balanced. If this is the case, then $\pi_\Gamma$ has the form
 \begin{equation}\label{eq:def_stat_distr_gamma}
   \pi_\Gamma(x)=M^{c}_\Gamma\prod_{i\colon S_i\in\Sp_\Gamma}\frac{c_i^{x_i}}{x_i!}\quad\text{for }x\in\Gamma,
 \end{equation}
 where $c$ is a positive complex balanced equilibrium of $(\G_\Gamma,\kappa_\Gamma)$ and $M^{c}_\Gamma$ is a normalising constant.
 \end{theorem}
 
 The proof is in Appendix \ref{sec:proofs}. It is shown in \cite{anderson:product-form} that the stationary distribution $\pi_\Gamma(x)$ is independent of the choice of complex balanced equilibrium $c$ of the $\Gamma$-system, provided that it is positive.
 We are now ready to derive stochastic versions of Theorem \ref{thm:complex_balanced}-\ref{thm:deficiency_zero}. In addition, we will show that a stochastically complexed balanced mass-action system is complex balanced and \emph{vice versa}. Hence, we will show that the deterministic and stochastic systems are intimately connected.
 The next corollary is an analogue of Theorem \ref{thm:complex_balanced}.

 \begin{corollary}\label{cor:stochastic_complex_balanced_positive}
  If a stochastic reaction system $(\G,K)$ is stochastically complex balanced then $\G$ is weakly reversible. Moreover, a mass-action system $(\G,\kappa)$ is stochastically complex balanced if and only if it is complex balanced. If this is case, then on every irreducible component $\Gamma$ there exists a unique stationary distribution $\pi_\Gamma$. Such $\pi_\Gamma$ is a complex balanced stationary distribution and it has the form \eqref{eq:definition_stat_prob}, where $c$ is a positive complex balanced equilibrium of $(\G,\kappa)$.
 \end{corollary}
 \begin{proof}
  If $\Gamma$ is positive, then $(\G_\Gamma, K_\Gamma)=(\G,K)$. Therefore, by Theorem \ref{thm:stochastic_complex_balanced} if $(\G,K)$ is stochastically complex balanced then $\G$ is weakly reversible. Moreover, if $K$ is mass-action kinetics with rate constants $\kappa$, it follows from Theorem \ref{thm:stochastic_complex_balanced} that there exists a complex balanced stationary distribution on $\Gamma$ if and only if $(\G,\kappa)$ is complex balanced. In this case, by Theorem \ref{thm:anderson}, a stationary distribution exists on every irreducible component and it is of the form \eqref{eq:definition_stat_prob}. By Theorem \ref{thm:stochastic_complex_balanced}, it is a complex balanced stationary distribution.
 \end{proof}

 Corollary \ref{cor:stochastic_complex_balanced_positive} might be considered a stochastic version of  Theorem \ref{thm:complex_balanced}, especially if \eqref{eq:limit_stationary} is taken to be equivalent to ``asymptotic stability'' for a deterministic equilibrium.
 Part of the corollary is known \cite{anderson:product-form} (see also Theorem \ref{thm:anderson}), and the whole corollary might therefore be considered as an extension of the result in  \cite{anderson:product-form} on mass-action systems.
  In this sense,  Theorem \ref{thm:stochastic_complex_balanced} provides an even  more general version, which deals with complex balanced subsystems of $(\G,\kappa)$.
 
 We now state the parallel versions of Theorem \ref{thm:deficiency_zero_iff}-\ref{thm:deficiency_zero} for the stochastic setting.
 \begin{corollary}\label{cor:stoch_deficiency_zero_iff}
  The mass-action system $(\G,\kappa)$ is stochastically complex balanced for any choice of $\kappa$ if and only if $\G$ is weakly reversible and its deficiency is zero.
 \end{corollary}
 \begin{proof}
  The result is an immediate consequence of Corollary \ref{cor:stochastic_complex_balanced_positive} and Theorem \ref{thm:deficiency_zero_iff}.
 \end{proof}

  \begin{theorem}\label{thm:stationary_boundary}
  Consider a stochastic reaction system $(\G,K)$, and assume the deficiency of $\G$ is zero. Let $x$ be a state in an irreducible component $\Gamma$ and let $y\to y'$ in $\R$. Then, $y\leq x$ only if $y\to y'$ is terminal. Moreover, if $K$ is mass-action kinetics, then on $\Gamma$ the stationary distribution has the form
  \begin{equation}\label{eq:stat_prob_terminal}
   \pi_\Gamma(x)=M^c_\Gamma\prod_{i\colon S_i\in\Sp^*}\frac{c_i^{x_i}}{x_i!}\quad\text{for }x\in\Gamma,
  \end{equation}
  where $c$ is a positive complex balanced equilibrium for the terminal system, and $M^c_\Gamma$ is a normalising constant. 
 \end{theorem}
 
 The proof is in Appendix \ref{sec:proofs}.

 \begin{theorem}\label{thm:stoch_deficiency_zero}
  Consider a stochastic reaction system $(\G,K)$, and assume that the deficiency of $\G$ is zero. Then the following statements hold:
  \begin{enumerate}[i)]
   \item if $\G$ is not weakly reversible, then there exist no positive irreducible components;
   \item if $\G$ is weakly reversible, then $\G$ is essential, and if $K$ is mass-action kinetics then there exists a unique stationary distribution on every irreducible component.
  \end{enumerate}  
 \end{theorem}

 The proof of the theorem is in Appendix \ref{sec:proofs}. In  case (i), Theorem \ref{thm:stationary_boundary} provides the form of the stationary distribution. Hence we have characterised the stationary distribution for any deficiency zero reaction system, irrespectively whether it is complex balanced or not.
 
 \begin{ex} 
  Consider the two stochastic mass-action systems
  $$A \cee{<=>[\kappa_1][\kappa_2]} B,\quad 10A \cee{<=>[\kappa_3][\kappa_4]} 10B\quad\text{and}\quad A \cee{<=>[\kappa_1][\kappa_2]} B,\quad 10A \cee{->[\kappa_3]} 0.$$
 The  behaviours of the two corresponding deterministic systems differ substantially, while the behaviours of the stochastic systems are equivalent on
  the irreducible components $\Gamma_\theta=\{x\in\NN^2\colon x_1+x_2=\theta\}$ with  $0\le\theta<10$ an integer. Indeed, in both cases the $\Gamma_\theta$-system is
  $$A \cee{<=>[\kappa_1][\kappa_2]} B,$$
  which is complex balanced (Theorem \ref{thm:deficiency_zero_iff}). It follows from Theorem \ref{thm:stochastic_complex_balanced} that the stationary distribution on  $\Gamma_\theta$ is 
  $$\pi_\theta(x_1,x_2)=M_\theta\frac{\kappa_2^{x_1}}{x_1!}\frac{\kappa_1^{x_2}}{x_2!}\quad\text{for }(x_1,x_2)\in\Gamma_\theta,$$
  for a suitable normalizing constant $M_\theta$. The stationary distributions are complex balanced, but since $\Gamma_\theta$ is not positive in either of the two networks, we cannot conclude that the systems are stochastically complex balanced. Indeed, they are not for some choice of rate constants (Corollary \ref{cor:stoch_deficiency_zero_iff}).
  
  Incidentally, note that the second network is not almost essential.
  \end{ex}

\section{Product-form Poisson-like stationary distributions}

 The above results draw  parallels between stochastic and deterministic reaction networks. If a mass-action system is (stochastically) complex balanced, then the stationary distribution on every irreducible component is a product-form Poisson-like distribution. Does the reverse statement  hold true too? If the stationary distribution is a product-form Poisson-like distribution on some, or all irreducible components, does it follow that the system is complex balanced?  
 In the spirit of the first part of the paper we would like to achieve a full characterisation of stochastic systems with product-form Poisson-like stationary distributions. However, even though the hypothesis of Theorem \ref{thm:main_thm} below is rather general, a full characterisation seems hard to achieve. 
 \begin{theorem}\label{thm:main_thm}
  Let $\G$ be an almost essential reaction network,  $\kappa\in\RR^k_+$  a vector of rate constants and $c\in\RR^n_+$ a  vector with positive entries. 
  The probability distribution $\pi_\Gamma\colon\Gamma\to(0,1]$, defined by \eqref{eq:definition_stat_prob} is a stationary distribution for the stochastic mass-action system $(\G,\kappa)$ for all irreducible components $\Gamma\subseteq \NN^n$ of $\G$ if and only if $c$ is a complex balanced equilibrium for $(\G,\kappa)$.
 \end{theorem}
 \begin{proof}
 By Theorem \ref{thm:anderson}, if $c>0$ is a complex balanced equilibrium for $(\G,\kappa)$, then the stationary distribution on all irreducible components $\Gamma\subseteq\NN^n$ is of the form \eqref{eq:definition_stat_prob}. 
 
 Oppositely, assume that \eqref{eq:definition_stat_prob} is the stationary distribution  on $\Gamma$ for the stochastic mass-action system $(\G,\kappa)$, for all irreducible components $\Gamma$. Since $\G$ is almost essential, there exists a constant $K$ such that any states $x$ with $\|x\|>K$ belongs to an irreducible component $\Gamma$. For any $x\in\NN^n$, such that 
 \begin{equation}\label{eq:lower_bound_x}
  \min_{S_i\in\Sp}x_i>\max_{y\to y'\in\R}(\|y\|_\infty+\|y'\|_\infty)+K,
 \end{equation}
 we have that $x\ge y$ and $x-y'+y\geq y$ for all  $y\to y'\in\R$. Then, since \eqref{eq:definition_stat_prob} is a stationary distribution and since $x$ and $x+y-y'$ are in the same irreducible component for all $y\to y'\in\R$,  we have from \eqref{eq:master}
 \begin{equation}\label{eq:polynomial_equality_first_step}
  \sum_{y\to y'\in\R}\pi_\Gamma(x+y-y')\kappa_{y\to y'}\frac{(x+y-y')!}{(x-y')!}=\pi_\Gamma(x)\sum_{y\to y'\in\R}\kappa_{y\to y'}\frac{x!}{(x-y)!},
 \end{equation}
 for all $x\in\Gamma$ satisfying \eqref{eq:lower_bound_x}. Further, using \eqref{eq:definition_stat_prob}, equation \eqref{eq:polynomial_equality_first_step} becomes
 $$\sum_{y\to y'\in\R}\frac{x!}{(x-y')!}\kappa_{y\to y'}c^{y-y'}=\sum_{y\to y'\in\R}\kappa_{y\to y'}\frac{x!}{(x-y)!},$$ 
 which, by rearranging terms, leads to
 \begin{equation}\label{eq:two_sides_equal}
  \sum_{y'\in\C}\frac{x!}{(x-y')!}\sum_{y\to y'\in\R}\kappa_{y\to y'}c^{y-y'}=\sum_{y'\in\C}\frac{x!}{(x-y')!}\sum_{y'\to y\in\R}\kappa_{y'\to y}.
 \end{equation}
 The equality holds for all $x\in\NN^n$ 
 satisfying \eqref{eq:lower_bound_x}, therefore the  polynomials on the two sides of \eqref{eq:two_sides_equal} are equal. 
 
 For any $y'\in\C$, let $p_{y'}(x)$ be the polynomial
 $$p_{y'}(x)=\frac{x!}{(x-y')!}.$$
 The monomial  with maximal degree in  $p_{y'}$ is $x^{y'}$, and these differ for all complexes $y'\in\C$.
 This  implies that $p_{y'}$, $y'\in\C$, are linearly independent on $\RR$, and thus, the polynomials on the two sides of \eqref{eq:two_sides_equal} are equal if and only if
 $$\sum_{y\in\C}\kappa_{y\to y'}c^{y-y'}=\sum_{y\in\C}\kappa_{y'\to y}\quad\text{for all }y'\in\C.$$
 Hence, $c$ is a complex balanced equilibrium for $(\G,\kappa)$ and the proof is completed.
\end{proof}

 \subsection{Relaxation of Assumptions in Theorem \ref{thm:main_thm}}
 
 To infer the existence of positive complex balanced equilibria in Theorem \ref{thm:main_thm}, the assumptions of the theorem could be weakened. Specifically,  it is only required that \eqref{eq:two_sides_equal} holds for a set of states whose geometry and cardinality allow us to conclude that the polynomials on the two sides of \eqref{eq:two_sides_equal} are the same. For \eqref{eq:two_sides_equal} to hold, we need $x$ to be in a irreducible component  and we require $x\ge y$ and $x-y'+y\geq y$ for all reactions $y\to y'\in\R$, as well as the stationary distribution evaluated in $x$ and  $x-y'+y$ to be of the form \eqref{eq:definition_stat_prob}. If a state $x$ satisfies this, we call it a \emph{good state}. 
 
 A more general condition than being almost essential could be chosen case by case and depends on the monomials appearing in \eqref{eq:two_sides_equal}. For example, if  the set of complexes  coincides with the set of species, then the polynomials in \eqref{eq:two_sides_equal} are linear and the existence of $n+1$ \emph{good states} in general position implies the existence of a positive complex balanced equilibrium. In general, let $d$ be the total degree of the polynomials in \eqref{eq:two_sides_equal}.  Then  it is sufficient to have $n$ lines in general position with more than $d+1$ good states on each of them. Therefore, to conclude that a system is complex balanced it is sufficient to check the behaviour of a finite number of states, lying on a finite number of irreducible components. However,  it follows from Examples \ref{ex:not_complex_balanced} and \ref{ex:not_complex_balanced_2} that the existence of arbitrarily many good states on a few irreducible components does not imply the existence of a positive complex balanced equilibrium in general. Finally, in order to postulate that the mass-action system is complex balanced, it is necessary that the vector $c$ appearing in  Theorem \ref{thm:main_thm} is the same for every irreducible component, as shown in Example \ref{ex:not_complex_balanced_different_c}.
 
 The following examples are also meant to give an idea of why it is hard to obtain a full characterization of stochastic mass-action systems with a product-form Poisson-like stationary distribution on some irreducible component.

 \begin{ex}\label{ex:not_complex_balanced} 
  Let $\rho\in\RR_+$ and let $\theta\geq2$ be an integer. Consider the stochastic mass-action system
  \begin{equation}\label{eq:ex}
   A \xrightarrow{\rho(\theta-1)} B\qquad 2B \xrightarrow[\phantom{\rho(\theta-1)}]{\rho} 2A,
  \end{equation}
  where $\kappa_1=\rho(\theta-1)$ and $\kappa_2=\rho$ are the rate constants. The reaction network is almost essential. It is shown in Appendix \ref{sec:calculations} that the stationary distribution on the irreducible component $\Gamma_\theta=\{x\in\NN^2\colon x_1+x_2=\theta\}$ has the form \eqref{eq:definition_stat_prob} with $c=(1,1)$,  namely
  \begin{equation}\label{eq:prod_distr_example}
   \pi_\theta(x_1,x_2)=M_\theta\frac{1}{x_1!x_2!}\quad\text{for }(x_1,x_2)\in\Gamma_\theta,
  \end{equation}
  where $M_\theta$ is a normalising constant. However, the mass-action system is not complex balanced as the reaction network is not weakly reversible (Theorem \ref{thm:complex_balanced}). In particular, by Theorem \ref{thm:main_thm}, not all irreducible components can have a stationary distribution of the form \eqref{eq:definition_stat_prob} with $c=(1,1)$. Trivially, the absorbing states $(0,0)$ and $(0,1)$ have it.
  
Additionally, we should point out that there is not an \emph{equivalent} system on $\Gamma$ (that is, a stochastic mass-action system with the same transition rate matrix on the states of $\Gamma$ as \eqref{eq:ex}) which is complex balanced. Consider the case $\theta=1$. Since the transition from $(0,2)$ to $(2,0)$ is possible according to \eqref{eq:ex}, any equivalent mass-action system must contain the reaction $2B\to2A$, with rate constant $\rho$. It can be further shown that any equivalent weakly reversible mass-action system must contain the connected component
  \begin{center}
  \begin{tikzpicture}
   %states
   \node[state] (A+B) at (0,0) {$A+B$};
   \node[state] (2B)  at (2,0.6) {$2B$};
   \node[state] (2A)  at (4,0) {$2A\;.$};
   %edges
   \path[->] 
    (2B)  edge node{$\rho$} (2A)
    (2A)  edge node{$\frac{\rho}{2}$} (A+B)
    (A+B) edge node{$\rho$} (2B);
  \end{tikzpicture}
  \end{center}
  This prevents the system from being complex balanced, since there is not a $c\in\R_+^2$  fulfilling \eqref{eq:cb_mass-action} for the three complexes $2B$, $2A$ and $A+B$.
 \end{ex}

  \begin{ex}
  Let $\rho_1,\rho_2,\rho_3\in\RR_+$ and let $\theta\geq2$ be an integer. Consider the modification of  Example \ref{ex:not_complex_balanced} given by
  $$A \cee{<=>[\rho_1(\theta-1)+\rho_2][\rho_2]} B\qquad
   2B \cee{<=>[\phantom{\theta-}\rho_1+\rho_3\phantom{()1}][\rho_3]} 2A,$$
  which is weakly reversible.
  If we let $\rho_2=0$ and $\rho_3=0$, then the system reduces to that of Example \ref{ex:not_complex_balanced} by removing the two reversible reactions. It can be shown that for any parameter choice, \eqref{eq:prod_distr_example} is still a stationary distribution on the irreducible component $\Gamma_\theta=\{x\in\NN^2\,:\,x_1+x_2=\theta\}$. However, for some choice of parameters the mass-action system is not complex balanced. This can be seen either by direct computation on the system of complex balance equations \eqref{eq:cb_mass-action} or by noting that the deficiency of the network is 1, so there must be a choice of parameters which prevents positive complex balanced equilibria by Theorem \ref{thm:deficiency_zero_iff}. It can be further shown that irreducible components different from $\Gamma$ do not possess a product-form Poisson-like stationary distribution.
 \end{ex}
 \begin{ex}\label{ex:not_complex_balanced_2}
  Consider the stochastic mass-action system with $\rho\in\RR_+$ and $\theta_1,\theta_2$ two positive integers,
   \begin{alignat*}{2}
   A &\xrightarrow[\phantom{\rho(\theta_1+\theta_2-1)}]{\rho\theta_1\theta_2} B&
   2B &\xrightarrow{\rho(\theta_1+\theta_2-1)} 2A\\[-0.3cm]
   3A &\xrightarrow[\phantom{\rho(\theta_1+\theta_2-1)}]{\rho} A+2B\qquad&
   2A+B &\xrightarrow[\phantom{\rho(\theta_1+\theta_2-1)}]{\rho} 3B.
  \end{alignat*}
  The reaction network is almost essential. For any $\theta\in\NN$, consider the irreducible component $\Gamma_\theta=\{x\in\NN^2\colon x_1+x_2=\theta+1\}$. Then $\pi_{\theta_1}$ and $\pi_{\theta_2}$, defined as in \eqref{eq:prod_distr_example}, are the (unique) stationary distributions on the irreducible components $\Gamma_{\theta_1}$ and $\Gamma_{\theta_2}$, respectively. For the relevant calculations see Appendix \ref{sec:calculations}. However, the mass-action system is not complex balanced, since the reaction network is not weakly reversible (Theorem \ref{thm:complex_balanced}).
  \end{ex}
 \begin{ex}\label{ex:not_complex_balanced_different_c}
  Theorem \ref{thm:main_thm} can be also used to compute the stationary distribution of a stochastic mass-action system which behaves as a complex balanced system on the irreducible components. Consider the weakly reversible (and therefore essential) stochastic mass-action system
  $$
   A\cee{<=>[\kappa_1][\kappa_2]} 2A\qquad
   A+B \cee{<=>[\kappa_3][\kappa_4]} 2A+B.
  $$
  On every irreducible component $\Gamma_\theta=\{x\in\NN^2\,:\,x_2=\theta\}$, $\theta\in\NN$, the associated continuous time Markov chain, which describes the evolution of the counts of $A$, has the same distribution as the process associated with
  $$A \cee{<=>[\kappa_1+\kappa_3\theta][\kappa_2+\kappa_4\theta]}2A,$$
  because the transition rates coincide. The latter system is complex balanced for any choice of rate constants. The stationary distribution has the form (Theorem \ref{thm:main_thm})
 $$  \pi_\theta(x)=M_\theta \frac{1}{x!}\left(\frac{\kappa_2+\kappa_4\theta}{\kappa_1+\kappa_3\theta}\right)^{\!\!x}$$
   for some positive constant $M_\theta$. The latter gives the stationary distribution of the original system as well. However, the rate of the Poisson distribution does depend on $\theta$, in which case the original system cannot be complex balanced (Corollary \ref{cor:stochastic_complex_balanced_positive}). For the same reason the example does not contradict Theorem  \ref{thm:main_thm}.
     \end{ex}
 
 \section{Applications}

 There are not many means to explicitly calculate the stationary distribution of a stochastic mass-action system. As an example, Theorem \ref{thm:stochastic_complex_balanced} can be used to determine the stationary distributions of mass-action systems like
 $$C\cee{<=>[\kappa_1][\kappa_2]}D,\quad 2A\cee{<=>[\kappa_3][\kappa_4]}2B,\quad A\cee{->[\kappa_5]}0.$$
 Indeed, for any irreducible component $\Gamma$ different from $\{0\}$, the $\Gamma$-system is given by
 $$C\cee{<=>[\kappa_1][\kappa_2]}D,$$
 which is weakly reversible and has deficiency zero, therefore it is complex balanced. Hence, the stationary distribution on $\Gamma$ has the form
 $$\pi_\Gamma(x)=M_\Gamma\frac{\kappa_2^{x_3}}{x_3!}\frac{\kappa_1^{x_4}}{x_4!}\quad\text{for }x\in\Gamma,$$
 where $x_3$ and $x_4$ denote the entries relative to $C$ and $D$, respectively. Alternatively, since the terminal system is given by
 $$C\cee{<=>[\kappa_1][\kappa_2]}D,\quad 2A\cee{<=>[\kappa_3][\kappa_4]}2B,$$
 Theorem \ref{thm:stationary_boundary} can be used to compute the stationary distribution. On every irreducible component $\Gamma$, it is given by 
 $$\pi_\Gamma(x)=\widetilde{M}_\Gamma\frac{(\sqrt{\kappa_4})^{x_1}}{x_1!}\frac{(\sqrt{\kappa_3})^{x_2}}{x_2!}\frac{\kappa_2^{x_3}}{x_3!}\frac{\kappa_1^{x_4}}{x_4!}\quad\text{for }\quad x\in\Gamma,$$
 which is equivalent to the previous formula since $x_1$ and $x_2$ are constantly $0$ on all irreducible components.

 If the system does not fulfil the conditions of Theorem \ref{thm:stochastic_complex_balanced} and neither can be cast as a birth-death process,  Theorem \ref{thm:main_thm} might be useful.   The following mass-action system is considered in \cite{anderson:lyapunov}:
 $$
   A\xrightarrow{\kappa_1}0\qquad
   0\xrightarrow{\kappa_2}2A.
 $$
 By Theorem \ref{thm:main_thm}, the stationary distribution cannot be Poisson. Indeed, it is given by the distribution of $Y=Y_1+2Y_2$, where $Y_1$ and $Y_2$ are two independent Poisson random variables with rates $\frac{\kappa_2}{\kappa_1}$ and $\frac{\kappa_2}{2\kappa_1}$, respectively. Hence,
 $$\pi(x)=e^{-\frac{3\kappa_2}{2\kappa_1}}\sum_{\substack{i,j\in\NN\\x=i+2j}}\frac{1}{i!j!}\left(\frac{\kappa_2}{\kappa_1}\right)^{i} \left(\frac{\kappa_2}{2\kappa_1}\right)^j.$$ 
 In \cite{anderson:lyapunov}, the following system is also considered:
 $$
  0\cee{<=>[\kappa_1][\kappa_2]}A\qquad
  2A\cee{<=>[\kappa_3][\kappa_4]}3A.
 $$
 It has the stationary distribution 
 $$\pi(x)=M\prod_{i=1}^x\frac{\theta_1[(i-1)(i-2)+\theta_2]}{i(i-1)(i-2)+\theta_3i}\quad\text{for }x\in\NN,$$
 where $\theta_1={\kappa_3}/{\kappa_4}$, $\theta_2={\kappa_1}/{\kappa_3}$, $\theta_3={\kappa_2}/{\kappa_4}$
 and $M=\pi(0)$ is a normalising constant. It is interesting  that $\pi(x)$ is a Poisson distribution if and only if $\theta_2=\theta_3$. In fact, and in accordance with our results, the mass-action system is complex balanced if and only if $\theta_2=\theta_3$.
 
 \section{Discussion}
 
 Corollary \ref{cor:stoch_deficiency_zero_iff} provides a characterisation of reaction networks that are stochastically complex balanced for any choice of rate constants. It is natural to wonder whether  a stationary distribution of the form \eqref{eq:definition_stat_prob} on some irreducible component $\Gamma$ for all choices of rate constants implies something specific about the $\Gamma$-system. If for specific form we intend deficiency zero and weakly reversible, this is not the case, as this is violated in Example \ref{ex:not_complex_balanced_different_c}. However, in Example \ref{ex:not_complex_balanced_different_c} the system might be described \emph{equivalently} by means of a weakly reversible deficiency zero system for any irreducible component. The question of whether this is always true remains open. We provide here two more examples.
 \begin{ex}\label{ex:equivalent_to_complex_balanced_finite_case}
  Consider the stochastic mass-action system
  $$
   2A\cee{->[\kappa_1]}2B\qquad
   A+3B\cee{->[\kappa_2]}3A+B.
  $$
  The underlying reaction network is considered in Figure \ref{figure}. On the irreducible component $\Gamma=\{(1,5),(3,3),(5,1)\}$, the Markov chain associated with the system has the same distribution as the Markov chain associated with
  $$
   2A\cee{<=>[\kappa_1][3\kappa_2]}2B,
  $$
  since the transition rates coincide. It is interesting to note that the dynamics of the two systems are different when they are deterministically modelled \cite{craciun:identifiability}. Due to Theorem \ref{thm:deficiency_zero_iff}, the latter system is complex balanced for any choice of rate constants. Therefore, by Theorem \ref{thm:main_thm}, the stationary distribution on $\Gamma$ has the form \eqref{eq:definition_stat_prob} on both systems for any choice of rate constants. The same argument does not hold, in this case, for the other irreducible components.
 \end{ex}
 \begin{ex}
  The same phenomenon as in Example \ref{ex:equivalent_to_complex_balanced_finite_case} is observed in the stochastic mass-action system
  $$
   2A\cee{->[\kappa_1]}3A+B\qquad
   A+3B\cee{->[\kappa_2]}2B.
  $$
  On the irreducible component $\Gamma=\{(x_1,x_2)\in\NN^2\,:\,x_1\geq2, x_1=x_2\}$, the Markov chain associated with the system has the same distribution as the Markov chain associated with
  $$
   2A\cee{<=>[\kappa_1][\kappa_2]}3A+B,
  $$
  since the transition rates coincide, and the latter network is weakly reversible and has deficiency zero. %Therefore, the original system has a stationary distribution of the form \eqref{eq:definition_stat_prob} on $\Gamma$, for any choice of rate constants, However, the reactions governing the dynamics of the original system on $\Gamma$ do not form a complex balanced system. 
 \end{ex}

 \begin{appendix}
 \section{Preliminary results}
 
 Here we state some preliminary results that will be needed in Appendix \ref{sec:proofs}.
 \begin{appxlemma}\label{lem:path}
  Let $\G$ be a reaction network. If $y_1\to y_2\to\dots\to y_q$ is a directed path in the reaction graph of $\G$, and $x\geq y_1$, then $x+y_q-y_1$ is accessible from $x$.
 \end{appxlemma}
 \begin{proof}
  First, note that 
  $$x+\sum_{i=1}^{q-1} (y_{i+1}-y_i)=x+y_q-y_1.$$
  It is sufficient to note that if $x\geq y_1$, then for any $1\leq j\leq q-1$, we have
  $$x+\sum_{i=1}^{j-1} (y_{i+1}-y_i)= x+y_j-y_1\geq y_j.$$
  This concludes the proof.
 \end{proof}
 \begin{appxlemma}\label{lem:weakly_reversible}
  Let $\Gamma$ be an irreducible component such that $\G_\Gamma$ has deficiency zero. Then, $\G_\Gamma$ is weakly reversible. In particular, if $\G$ has deficiency zero, $\G_\Gamma$ has deficiency zero and is weakly reversible for every irreducible component $\Gamma$.
 \end{appxlemma}
 \begin{proof} If $\R_\Gamma$ is empty then $\G_\Gamma$ is weakly reversible and there is nothing to prove. Otherwise, if $\R_\Gamma$ is non-empty, let $y_1\to y_1'\in\R_\Gamma$. By hypothesis, there exists a state $x$ in $\Gamma$ with $x\geq y_1$. This means that $x+\xi_{y_1\to y_1'}$ is accessible from $x$. Moreover, since $x$ belongs to an irreducible component $\Gamma$, we have that $x$ is accessible from $x+\xi_{y_1\to y_1'}$ as well, which implies that
  $$x=x+\sum_{j=1}^q \xi_{y_j\to y_j'},$$
  for a certain choice of $\xi_{y_j\to y_j'}$. In particular, $\sum_{j=1}^q \xi_{y_j\to y_j'}=0$. By the hypothesis of deficiency zero, it follows that $\sum_{j=1}^q d_{y_j\to y_j'}=0$,
  because $\varphi$, defined in \eqref{eq:def_phi}, is an isomorphism between the spaces $D$ and $S$ associated with $\G_\Gamma$. Therefore, 
  $$\sum_{j=1}^q (e_{y_j'}-e_{y_j})=\sum_{y\in\C_\Gamma} \alpha_y e_y =0,$$
  for some integers $\alpha_y$. Since the vectors $e_y$ are linearly independent, $\alpha_y=0$ for all $y\in\C_\Gamma$.
  Hence, each $e_y$ that appears in the sum, must appear at least twice, once with coefficient $1$, once with $-1$. Consequently, by iteratively reordering the terms $d_{y_j\to y_j'}$, the reactions $(y_j\to y_j')_{j=1}^{q}$ form a union of directed closed paths in the reaction graph of $\G$. In particular, the reaction $y_1\to y_1'$ is contained in a closed directed path of the reaction graph of $\G_\Gamma$, and since this is true for every reaction in $\R_\Gamma$, $\G_\Gamma$ is weakly reversible. We conclude the proof by Lemma \ref{lem:subnet}, since if $\G$ has deficiency zero, so does every subnetwork of $\G$.
 \end{proof}
 \begin{appxlemma}\label{lem:reactions_gamma}
  Let $\G$ be a weakly reversible reaction network, and let $\Gamma$ be an irreducible component. Then, for any complex $y'\in\C_\Gamma$ we have
  \begin{align*}
   \{y\in\C\colon y\to y'\in\R\} &=\{y\in\C_\Gamma\colon y\to y'\in\R_\Gamma\}, \\
   \{y\in\C\colon y'\to y\in\R\} &=\{y\in\C_\Gamma\colon y'\to y\in\R_\Gamma\}.
  \end{align*}
 \end{appxlemma}
 \begin{proof}
  One inclusion is trivial, since $\R_\Gamma\subseteq\R$. For the other inclusion, fix $y'\in\C_\Gamma$. Suppose that there exists $x\in\Gamma$ with $x\geq y'$. It follows that any reaction $y'\to y\in\R$ is active on $\Gamma$, and therefore is contained in $\R_\Gamma$. Moreover, since $\G$ is weakly reversible, for any reaction in $\R$ of the form $y\to y'$, there exists a directed path in the reaction graph of $\G$ from $y'$ to $y$. Hence, by Lemma \ref{lem:path}, $x+y-y'$ is accessible from $x$, which implies that $x+y-y'$ is in $\Gamma$ and that $y\to y'$ is in $\R_\Gamma$, since $x+y-y'\geq y$. Therefore, to conclude the proof it suffices to prove that there exists $x\in\Gamma$ with $x\geq y'$.
 
  If it were no $x\in\Gamma$ with $x\geq y'$, then no reaction of the form $y'\to y$ would be in $\R_\Gamma$. Since $y'\in\C_\Gamma$, there  exists a reaction of the form $y\to y'$. This means that there is $\tilde{x}\in\Gamma$, such that $\tilde{x}\geq y$. Hence, $\tilde{x}+y'-y$ is in $\Gamma$ with $\tilde{x}+y'-y\geq y'$, which concludes the proof.   
 \end{proof}
 
%  \newpage 
 \section{Proofs}\label{sec:proofs}

 \subsection{Proof of Theorem \ref{thm:complex_balanced}}
 
 It is proven in \cite{horn:general_mass_action} that if a deterministic reaction system $(\G,K)$ is complex balanced, then $\G$ is weakly reversible. By \cite{horn:general_mass_action}, we also know that if $K$ is mass-action kinetics, then all positive equilibria are complex balanced, and there exists exactly one positive equilibrium in each stoichiometric compatibility class, which is locally asymptotically stable. Therefore, to conclude the proof we only need to prove that in a complex balanced mass-action system $(\G,\kappa)$, the eventual  equilibria on the boundary of $\RR^n$ are also complex balanced.
 
 First of all note that any subsystem $(\G_L,\kappa_L)$ of $(\G,\kappa)$ corresponding to a linkage classes $L$ of $\G$ is complex balanced. Indeed, the projection of a positive complex balanced equilibrium of $(\G,\kappa)$ onto the space of the species of $L$ satisfies \eqref{eq:cb_mass-action} for any complex of $\G_L$, hence it is a positive complex balanced equilibrium of $(\G_L,\kappa_L)$.
 
 Let $c$ be an equilibrium point on the boundary. Consider a linkage class $L$ of $\G$, and assume that $c_S>0$ for any species $S$ appearing in the linkage class. Then, the projection of $c$ onto the species of $L$ is a positive equilibrium of $(\G_L,\kappa_L)$, and therefore complex balanced. It follows that $c$ satisfies \eqref{eq:cb_mass-action} for any complex of $L$.
 Oppositely, assume that there exists a species appearing in the linkage class $L$, such that $c_S=0$ (this can only happen on a boundary state). Remember that by mass-action kinetics, all the rates of reactions whose source complex contains $S$ are zero. In particular, all the rates of reactions degrading $S$ are zero. Consider a complex $y$ in $L$ that contains $S$. By weakly reversibility, there exists a reaction $y'\to y$ in $L$. If $y'$ contains $S$, then $\lambda_{y'\to y}(c)=0$. If $y'$ does not contain $S$, then the reaction $y'\to y$ produces $S$.  Since the rate of all reactions degrading $S$ is zero at $c$ and $c$ is an equilibrium, then $\lambda_{y'\to y}(c)$ must be zero as well. By mass-action kinetics, this means that there exists a species $S'\neq S$ such that $S'$ appears in $y'$ and $c_{S'}=0$. By iteratively applying the same argument with the new species $S'$ and by weakly reversibility, we obtain that $\lambda_{y\to y'}(c)=0$ for any reaction $y\to y'$ in $L$. It follows that $c$ satisfies \eqref{eq:cb_mass-action} for any complex in $L$, since the equation reduces to $0=0$. Equation \eqref{eq:cb_mass-action} is therefore satisfied for any complex of $\G$ and $c$ is a complex balanced equilibrium. This concludes the proof.
 \qed
 
 \subsection{Proof of Theorem \ref{thm:equilibrium_boundary}}
 
  By \cite[Theorem 6.1.2]{feinberg:review}, if $x\in\RR_0^n$ is an equilibrium point and $y\to y'\in\R$, then $\supp y\subseteq \supp x$ only if $y\to y'$ is terminal. Moreover, if $\supp y\subseteq \supp x$, then $\supp \widetilde{y}\subseteq \supp x$ for every complex $\widetilde{y}$ of $\G_y=(\Sp_y,\C_y,\R_y)$.
  
  Now, suppose that $K$ is mass-action kinetics with rate constants $\kappa$, and that $\supp y\subseteq \supp x$ with $y\to y'\in\R$ (and therefore $y\to y'\in\R^*$). Consider
  $$\widetilde{\R}=\{\widetilde{y}\to \widetilde{y}'\in\R\colon \supp \widetilde{y}\subseteq \supp x\}.$$
  By the first part of the statement, the reaction graph of the subnetwork $\widetilde{\G}=(\widetilde{\Sp},\widetilde{\C},\widetilde{\R})$ is a union of terminal strongly connected components of $\G$, and therefore $\widetilde{\G}$ is weakly reversible. Moreover, by Lemma \ref{lem:subnet}, the deficiency of $\widetilde{\G}$ is 0. It is not hard to see that the canonical projection of $x$ onto the space of the species $\widetilde{\Sp}$ is a positive equilibrium point of $(\widetilde{\G},\widetilde{\kappa})$, and therefore complex balanced by Theorem \ref{thm:deficiency_zero_iff}. The proof is concluded by \eqref{eq:cb_mass-action} and by noting that, for any complex $\widetilde{y}\in\C_y$,
  \begin{align*}
   \{\widetilde{y}'\in\widetilde{\C}\colon \widetilde{y}\to \widetilde{y}'\in\widetilde{\R}\} &=\{\widetilde{y}'\in\C_y\colon \widetilde{y}\to \widetilde{y}'\in\R_y\}, \\
   \{\widetilde{y}'\in\widetilde{\C}\colon \widetilde{y}'\to \widetilde{y}\in\widetilde{\R}\} &=\{\widetilde{y}'\in\C_y\colon \widetilde{y}'\to \widetilde{y}\in\R_y\}.
  \end{align*}
  \qed

 \subsection{Proof of Proposition \ref{prop:terminal_strong_linkage_classes}}
 
  If $\R_\Gamma$ is empty there is nothing to prove. Suppose that this is not the case.
  Since $\G_\Gamma$ has deficiency zero, by Lemma \ref{lem:weakly_reversible}, it is weakly reversible. For any $y\to y'\in\R_\Gamma$, by definition there exists $x\in\Gamma$ such that $x\geq y$, which in turn implies $x+y'-y\geq y'$. Therefore, for any directed path in the reaction graph of $\G$ that starts with $y\to y'\in\R_\Gamma$, all the reactions in the path belong to $\R_\Gamma$, by definition of $\R_\Gamma$. Since $\G_\Gamma$ is weakly reversible, this can only happen if $\R_\Gamma\subseteq\R^*$, and this proves the first part of the statement.
  To conclude the proof, note that if the deficiency of $\G$ is zero, then by Lemma \ref{lem:subnet} the deficiency of $\G_\Gamma$ is zero as well.\qed
 
 \subsection{Proof of Theorem \ref{thm:stochastic_complex_balanced}}
 
  For the first part of the statement, consider a continuous-time Markov chain $C_\Gamma(t)$ with state space $\Gamma\times\C$ and transition rate from $(x,y)$ to $(x+y'-y,y')$ given by $\lambda_{y\to y'}(x)$ if $y\to y'\in\R_\Gamma$, and zero otherwise. The master equation for $C_\Gamma(t)$ is 
  \begin{equation*}
   \sum_{y\in\C_\Gamma}\tilde{\pi}(x-y'+y,y)\lambda_{y\to y'}(x-y'+y)=\sum_{y\in\C_\Gamma}\tilde{\pi}(x,y')\lambda_{y'\to y}(x)\quad\forall y'\in\C, x\in\Gamma,
  \end{equation*}
  with the convention that $\lambda_{y\to y'}(x)=0$ if $y\to y'\notin\R_\Gamma$. By Definition \ref{defi:complex_balanced_stationary_distribution}, a stationary distribution for $C_\Gamma(t)$ exists and it is of the form $\tilde{\pi}(x,y)=M\pi(x)$, for a suitable normalising constant $M$. Since $\pi(x)$ is positive for any $x\in\Gamma$ (because it is a stationary distribution on an irreducible component), then by standard Markov chain theory, we have that for any two states $(x_1,y_1),(x_2,y_2)\in\Gamma\times\C$, if $(x_2,y_2)$ is accessible from $(x_1,y_1)$, then $(x_1,y_1)$ is accessible from $(x_2,y_2)$. Fix $y\to y'\in\R_\Gamma$ and $x\in\Gamma$ with $x\geq y$. Then, a directed path from $(x+y'-y,y')$ to $(x,y)$ exists in the graph associated with $C_\Gamma(t)$. The second components of the form $y$ of the states in the path, by construction, determine a directed path in the reaction graph of $\G_\Gamma$ from $y'$ to $y$. Hence, any reaction $y\to y'\in\R_\Gamma$ is contained in a closed directed path, which means that $\G_\Gamma$ is weakly reversible.
  
  Assume now that $K$ is mass-action kinetics with rate constants $\kappa$ and that $c$ is a positive complex balanced equilibrium of $(\G,\kappa)$. Then, by Theorem \ref{thm:anderson}, there exists a (unique) stationary distribution on $\Gamma$ of the form \eqref{eq:definition_stat_prob}. If a species $S_j$ is not in $\Sp_\Gamma$, then the value of $x_j$ is constant for any $x\in\Gamma$, and \eqref{eq:def_stat_distr_gamma} can be obtained from \eqref{eq:definition_stat_prob} by modifying the normalising constant.
 
  By Theorem \ref{thm:complex_balanced} and Lemma \ref{lem:reactions_gamma}, we have that
  $$\sum_{y\in\C_\Gamma}c^{y-y'}\kappa_{y\to y'}=\sum_{y\in\C_\Gamma}\kappa_{y'\to y}\quad\forall y'\in\C_\Gamma,$$
  with $\kappa_{y\to y'}=0$ if $y\to y'\notin\R_\Gamma$. Therefore, for any $y'\in\C_\Gamma$ and $x\in\Gamma$,
  $$\frac{1}{(x-y')!}\sum_{y\in\C_\Gamma}c^{x+y-y'}\kappa_{y\to y'}\mathbbm{1}_{\{x\geq y'\}}\\
   =\frac{1}{(x-y')!}\sum_{y\in\C_\Gamma}c^x\kappa_{y'\to y}\mathbbm{1}_{\{x\geq y'\}},$$
  which leads to \eqref{eq:stochastic_complex_balanced}, since $\pi$ is of the form \eqref{eq:definition_stat_prob}. 
 
  To prove the converse we first introduce a new stochastic mass-action system $(\hat{\G}_\Gamma, \hat{\kappa}_\Gamma)$, which is given by the reactions of the form
  $$y+S_y\to y'+S_{y'}\quad\text{with }y\to y'\in\R_\Gamma,$$
  where $S_{y}$ are fictitious species in one to one correspondence with the complexes $\C_\Gamma$. The rate constant of the reaction $y+S_y\to y'+S_{y'}$ is given by $\kappa_{y\to y'}$. It is not difficult to see that the sum of the fictitious species is conserved for any possible trajectory. Moreover, since any directed path $y_1\to y_2\to\dots y_q$ in the reaction graph of $\G$ corresponds to a directed path $y_1+S_{y_1}\to y_2+S_{y_2}\to\dots y_q+S_{y_q}$ in the reaction graph of $\hat{\G}_\Gamma$, we have that $\hat{\G}_\Gamma$ is weakly reversible by the first part of the proof.
 
  Consider the set
  $$\Upsilon=\{(x,\hat{x})\in\NN^n\times\NN^m:x\in\Gamma, \|\hat{x}\|_1=1\}.$$
  Every state in $\Upsilon$ is of the form $(x,S_y)\in\NN^{n+m}$, where $x\in\Gamma$ and $S_y$ is considered as the vector in $\NN^m$ with entry 1 in the position corresponding to the species $S_y$ and 0 otherwise. Since $\Gamma$ is an irreducible component of $\G$ and the sum of the fictitious species is conserved, no state outside $\Upsilon$ is accessible from any state in $\Upsilon$, according to $\hat{\G}_\Gamma$. Moreover, the master equation on $\Upsilon$ can be written as
  \begin{multline}\label{eq:master_equation_on_Upsilon}
   \sum_{y\in\C_\Gamma}\hat{\pi}(x-y'+y,S_y)\kappa_{y\to y'}\frac{(x-y'+y)!}{(x-y')!}\mathbbm{1}_{\{x\geq y'\}}\\
   =\sum_{y\in\C_\Gamma}\hat{\pi}(x,S_{y'})\kappa_{y'\to y}\frac{x!}{(x-y')!}\mathbbm{1}_{\{x\geq y'\}}\quad\forall y'\in\C, x\in\Gamma.
  \end{multline}
  If we choose $\hat{\pi}(x,\hat{x})=M\pi(x)$ for some positive constant $M$, then the master equation \eqref{eq:master_equation_on_Upsilon} is satisfied due to Definition \ref{defi:complex_balanced_stationary_distribution}. Therefore, if $M$ is chosen as a suitable normalising constant, $\hat{\pi}(x,z)=M\pi(x)$ is a stationary distribution on $\Upsilon$.

  Consider the linear homomorphism $\varphi$ as defined in \eqref{eq:def_phi}, for the reaction network $\hat{\G}_\Gamma$. Let $|\cdot|$ denote the cardinality of a set, and note that $|\hat{\C}_\Gamma|=|\C_\Gamma|=m_\Gamma$. For any vector $e_y$ of the basis of $\RR^{m_\Gamma}$, we have $\varphi(e_y)=(y,S_y)$. Since the vectors $(y,S_y)$ with $y\in\C_\Gamma$ are linear independent, $\varphi$ is an isomorphism and the deficiency of $\hat{\G}_\Gamma$ is 0.
 
  Since $\hat{\G}_\Gamma$ is a deficiency zero weakly reversible reaction network, it follows from Theorem \ref{thm:deficiency_zero_iff} that the mass-action system $(\hat{\G}_\Gamma, \kappa)$ is complex balanced. Therefore, by Theorem \ref{thm:anderson}, we have that $\hat{\pi}$ has the form
  $$\hat{\pi}(x,\hat{x})=M^{(c,\hat{c})}_{\hat{\Gamma}}\frac{c^x}{x!}\frac{\hat{c}^{\hat{x}}}{\hat{x}!},$$
  for a positive complex balanced equilibrium $(c,\hat{c})$, on any irreducible component $\hat{\Gamma}$ contained in $\Upsilon$. Since $\hat{\pi}(x,\hat{x})=M\pi(x)$ does not depend on $\hat{x}$, we have
  $$\hat{\pi}(x,\hat{x})=M^c_{\Gamma}\frac{c^x}{x!},$$
  for any $(x,\hat{x})\in\Upsilon$. 
 
  Fix a complex $y' \in\C_\Gamma$. Since $\G_\Gamma$ is weakly reversible, there exists a reaction $y'\to y$ that is active on $\Gamma$. Fix $x\in\Gamma$ such that $x\geq y'$. Then for any $y\to y'\in\R_\Gamma$ we have $x-y'+y\geq y$. If we plug the formula for $\hat{\pi}(x,\hat{x})$ in \eqref{eq:master_equation_on_Upsilon} for our choice of $x$ and $y'$, we obtain
  \begin{equation*}
   \sum_{y\in\C_\Gamma}M^c_{\Gamma}\frac{c^{x-y'+y}}{(x-y'+y)!}\kappa_{y\to y'}\frac{(x-y'+y)!}{(x-y')!}=\sum_{y\in\C_\Gamma}M^c_{\Gamma}\frac{c^x}{x!}\kappa_{y'\to y}\frac{x!}{(x-y')!},
  \end{equation*}
  which leads to
  $$\sum_{y\in\C_\Gamma}c^{y-y'}\kappa_{y\to y'}=\sum_{y\in\C_\Gamma}\kappa_{y'\to y}.$$
  The proof is concluded by the fact that the above holds for any fixed $y'\in\C_\Gamma$, which means that $c$ is a positive complex balanced equilibrium of $(\G_\Gamma,\kappa_\Gamma)$.\qed

 \subsection{Proof of Theorem \ref{thm:stationary_boundary}}
  
  By Lemma \ref{lem:weakly_reversible}, $\G_\Gamma$ is weakly reversible. Moreover, for $y\to y'\in\R_\Gamma$, if $x\geq y$ then $x+y'-y\geq y'$. This implies that for any directed path in the reaction graph of $\G$ that starts with $y\to y'\in\R_\Gamma$, all the reactions in the path belong to $\R_\Gamma$, by definition of $\R_\Gamma$. Since $\G_\Gamma$ is weakly reversible, every directed path in the reaction graph of $\G$ that starts with $y\to y'\in\R_\Gamma$ is contained in a closed directed path. This implies that $\R_\Gamma\subseteq\R^*$, and proves the first part of the statement.
  
  Now assume that $K$ is mass-action kinetics with rate constants $\kappa$. If the deficiency of $\G$ is zero, then by Lemma \ref{lem:subnet} the deficiency of the terminal network is zero as well. Moreover, $\G^*$ is weakly reversible by definition, thus by Theorem \ref{thm:deficiency_zero_iff} $(\G^*,\kappa^*)$ is complex balanced for any choice of rate constants $\kappa^*$.
  
  Let $X(t)$ be the stochastic process associated with $(\G,\kappa)$. By the first part of the statement, on $\Gamma$ only terminal reactions take place and these involve a subset of the species only. Without loss of generality, we can assume that $\Sp^*$ is constituted by the first $n^*$ species of $\Sp$. Therefore, $\Gamma$ is of the form $\Gamma^*\times \{v\}$, with $\Gamma^*\subseteq\R^{n^*}$ and $v\in\RR^{n-n^*}$. Moreover, we have that on $\Gamma^*$, the projection $X^*(t)=(X_1(t),\dots,X_{n^*}(t))$ is distributed as the process associated with $(\G^*,\kappa^*)$, for which $\Gamma^*$ is an irreducible component. Let $c$ be a positive complex balanced equilibrium for $(\G^*,\kappa^*)$. Hence, by Theorem \ref{thm:anderson} or Corollary \ref{cor:stochastic_complex_balanced_positive}, the stationary distribution of the process $X(t)=(X^*(t),v)$ on $\Gamma$ is of the form \eqref{eq:stat_prob_terminal}.\qed

 \subsection{Proof of Theorem \ref{thm:stoch_deficiency_zero}}
 
  For the first part, we prove that if an irreducible component $\Gamma$ is positive, then $\G$ is weakly reversible. This simply follows from Lemma \ref{lem:weakly_reversible}: indeed, by the lemma, $\G_\Gamma$ is weakly reversible and since $\Gamma$ is positive, $\G_\Gamma=\G$.
  
  To prove the second part, we have to show that a weakly reversible reaction network is essential, and this is done in \cite{craciun:dynamical}. Moreover, a deficiency zero weakly reversible mass-action system is complex balanced, and the proof is concluded by Theorem \ref{thm:anderson} or Corollary \ref{cor:stochastic_complex_balanced_positive}. \qed

  \section{Calculations for Examples \ref{ex:not_complex_balanced} and \ref{ex:not_complex_balanced_2}}\label{sec:calculations}
  
  In Example \ref{ex:not_complex_balanced}, we claim that the stationary distribution on the irreducible component $\Gamma_\theta=\{x\in\NN^2\colon x_1+x_2=\theta\}$ has the form 
  \begin{equation*}
   \pi_\theta(x_1,x_2)=M_\theta\frac{1}{x_1!x_2!}\quad\text{for }(x_1,x_2)\in\Gamma_\theta.
  \end{equation*}
  To prove this, it is sufficient to show that $\pi_\theta$ satisfies the master equation for every point $(x_1,x_2)$ of $\Gamma_\theta$. The master equation on $(x_1,x_2)$ is given by
  \begin{multline*}
   \rho\pi_\theta(x_1+1,x_2-1)(\theta-1)(x_1+1)\mathbbm{1}_{\{x_2\geq1\}}+\rho\pi_\theta(x_1-2,x_2+2)\frac{(x_2+2)!}{x_2!}\mathbbm{1}_{\{x_1\geq2\}}\\
   =\rho\pi_\theta(x_1,x_2)\pr{(\theta-1)x_1\mathbbm{1}_{\{x_1\geq1\}}+\frac{(x_2)!}{(x_2-2)!}\mathbbm{1}_{\{x_2\geq2\}}}.
  \end{multline*}
 By plugging in the formula for $\pi_\theta$ and after dividing by $\rho$ and $M_\theta$ we obtain
 $$\frac{1}{x_1!x_2!}[x_2(\theta-1)+x_1(x_1-1)]=\frac{1}{x_1!x_2!}[x_1(\theta-1)+x_2(x_2-1)].$$
 If we multiply by $x_1!x_2!$ and substitute $\theta=x_1+x_2$, it follows that
 $$x_2(x_1+x_2-1)+x_1(x_1-1)=x_1(x_1+x_2-1)+x_2(x_2-1),$$
 that is
 $$x_1x_2+x_2^2-x_2+x_1^2-x_1=x_1^2+x_1x_2-x_1+x_2^2-x_2,$$
 which always holds true because the terms cancel each other.

 In Example \ref{ex:not_complex_balanced_2}, we change the notation to $\Gamma_\theta=\{x\in\NN^2\colon x_1+x_2=\theta+1\}$. Then we claim that the stationary distributions on the irreducible components $\Gamma_{\theta_1}$ and $\Gamma_{\theta_2}$ are $\pi_{\theta_1}$ and $\pi_{\theta_2}$, respectively, where as before 
  \begin{equation*}
   \pi_\theta(x_1,x_2)=M_\theta\frac{1}{x_1!x_2!}\quad\text{for }(x_1,x_2)\in\Gamma_\theta.
  \end{equation*}
  We prove that $\pi_{\theta_1}$ is the stationary distribution on $\Gamma_{\theta_1}$. The case with $\theta_2$ is analogue. We prove the result by consider the master equation for $\pi_{\theta_1}$ on a point $(x_1,x_2)\in\Gamma_{\theta_1}$, which is as following:
 \begin{multline*}
  \rho\pi_{\theta_1}(x_1+1,x_2-1)\theta_1\theta_2(x_1+1)\mathbbm{1}_{\{x_2\geq1\}}\\
  +\rho\pi_{\theta_1}(x_1-2,x_2+2)(\theta_1+\theta_2-1)\frac{(x_2+2)!}{x_2!}\mathbbm{1}_{\{x_1\geq2\}}\\
  +\rho\pi_{\theta_1}(x_1+2,x_2-2)\frac{(x_1+2)!}{(x_1-1)!}\mathbbm{1}_{\{x_1\geq1,x_2\geq2\}}\\
  +\rho\pi_{\theta_1}(x_1+2,x_2-2)\frac{(x_1+2)!(x_2-2)!}{x_1!(x_2-3)!}\mathbbm{1}_{\{x_1\geq0,x_2\geq3\}}\\
   =\rho\pi_{\theta_1}(x_1,x_2)\theta_1\theta_2x_1\mathbbm{1}_{\{x_1\geq1\}}+\rho\pi_{\theta_1}(x_1,x_2)(\theta_1+\theta_2-1)\frac{x_2!}{(x_2-2)!}\mathbbm{1}_{\{x_2\geq2\}}\\
   +\rho\pi_{\theta_1}(x_1,x_2)\frac{x_1!}{(x_1-3)!}\mathbbm{1}_{\{x_1\geq3\}}
   +\rho\pi_{\theta_1}(x_1,x_2)\frac{x_1!x_2!}{(x_1-2)!(x_2-1)!}\mathbbm{1}_{\{x_1\geq2,x_2\geq1\}}.
 \end{multline*}
 As we did for the previous calculations, we plug in the expression for $\pi_{\theta_1}$, then divide by $M_{\theta_1}$, $\rho$ and multiply by $x_1!x_2!$. We obtain
  \begin{multline*}
  \theta_1\theta_2x_2+(\theta_1+\theta_2-1)x_1(x_1-1)+x_1x_2(x_2-1)+x_2(x_2-1)(x_2-2)\\
   =\theta_1\theta_2x_1+(\theta_1+\theta_2-1)x_2(x_2-1)+x_1(x_1-1)(x_1-2)+x_1(x_1-1)x_2.
 \end{multline*}
 Finally, by substituting $\theta_1$ with $x_1+x_2-1$ and by performing the calculations, we obtain $0=0$, which means that the above equation is satisfied.
 \end{appendix}

\end{document}